\documentclass[11pt,reqno]{amsart}
\usepackage{fullpage,amsfonts,amsmath,amssymb}
\usepackage{amsfonts,amsmath,amssymb,fullpage,graphicx}
\usepackage{verbatim}
\usepackage{hyperref}

\newcommand{\what}{\widehat}%
\newcommand{\wtilde}{\widetilde}%
\newcommand{\R}{\mathbb R}%
\newcommand{\C}{\mathbb C}%
\newcommand{\Z}{\mathbb Z}%
\newcommand{\N}{\mathbb N}%

\newtheorem{theorem}{Theorem}[section]
\newtheorem{lemma}[theorem]{Lemma}
\newtheorem{proposition}[theorem]{Proposition}
\newtheorem{corollary}[theorem]{Corollary}

\theoremstyle{definition}
\newtheorem{definition}[theorem]{Definition}

\theoremstyle{definition}
\newtheorem{remark}[theorem]{Remark}

\numberwithin{equation}{section}

\begin{document}
\title{Ramanujan's master theorem for Sturm-Liouville operator}

\author {Jotsaroop Kaur}
\address{Department of Mathematics, Indian Institute of Science Education and Reseach, Mohali, India.}
\email{jotsaroop@iisermohali.ac.in}
\author{Sanjoy Pusti}
\address{Department of Mathematics, Indian Institute of Technology, Bombay, India.}
\email{sanjoy@math.iitb.ac.in}

%\address{IIT Kanpur}
%\email{spusti@iitk.ac.in}
\thanks{The second author is supported partially by SERB, MATRICS, MTR/2017/000235.}

\begin{abstract}
In this paper we prove an analogue of Ramanujan's master theorem in the setting of Sturm Liouville operator.
\end{abstract}
\maketitle
\section{Introduction}
Ramanujan's Master theorem (\cite{Hardy}) states that if a function $f$ can be expanded around $0$ in a power series of the form $$f(x)=\sum_{k=0}^\infty (-1)^k a(k) x^k,$$ then 
\begin{equation}\label{eqn-11}
\int_0^\infty f(x) x^{-\lambda-1}\,dx=-\frac{\pi}{\sin\pi\lambda} a(\lambda).
\end{equation}

One needs some assumptions on the function $a$, as the theorem is not true for $a(\lambda)=\sin\pi\lambda$. Hardy provides a rigorous statement of the theorem above as:

Let $A, p, \delta$ be real constants such that $A<\pi$ and $0<\delta\leq 1$. Let $\mathcal H(\delta)=\{\lambda\in\C\mid \Re\lambda>-\delta\}$. Let $\mathcal H(A,p,\delta)$ be the collection of all holomorphic functions $a:\mathcal H(\delta)\rightarrow \mathbb C$ such that  $$ |a(\lambda)|\leq C e^{-p (\Re\lambda) + A |\Im\lambda|} \text{ for all } \lambda\in \mathcal H(\delta).$$ 
\begin{theorem}[Ramanujan's Master theorem, Hardy \cite{Hardy}] Suppose $a\in\mathcal H(A,p,\delta)$. Then 
\begin{enumerate}
\item The power series $f(x)=\sum_{k=0}^\infty (-1)^k a(k) x^k$ converges for $0<x<e^p$ and defines a real analytic function on that domain.
\item Let $0<\sigma<\delta$. Then for $0<x<e^p$ we have $$f(x)=\frac{1}{2\pi i}\int_{-\sigma-i\infty}^{-\sigma + i\infty}\frac{-\pi}{\sin \pi\lambda} a(\lambda) x^\lambda\, d\lambda.$$The integral on the right side of the equation above converges uniformly on compact subsets of $[0, \infty)$ and is independent of $\sigma$.

\item Also $$\int_0^\infty f(x) x^{-\lambda-1}\,dx=-\frac{\pi}{\sin \pi\lambda} a(\lambda),$$ holds for the extension of $f$ to $[0, \infty)$ and for all $\lambda\in\mathbb C$ with $0<\Re\lambda<\delta$.
\end{enumerate}

\end{theorem}
This theorem can be thought of as an interpolation theorem, which reconstructs the values of $a(\lambda)$ from it's given values at $a(k), k\in \N\cup \{0\}$. In particular if $a(k)=0$ for all $k\in \N\cup \{0\}$, then $a$ is identically $0$.
We can rewrite the  theorem above in terms of Fourier series and Fourier transform as follows:
\begin{theorem}\label{eigenexpansion} Suppose $a\in\mathcal H(A,p,\delta)$. Then 
\begin{enumerate}
\item The Fourier series $f(z)=\sum_{k=0}^\infty (-1)^k a(k) e^{-ikz}$ converges for $\Im z<p$ and defines a holomorphic function on that domain.
\item Let $0<\sigma<\delta$. Then for $0<t<p$ we have $$f(it)=\frac{1}{2\pi i}\int_{-\sigma-i\infty}^{-\sigma + i\infty}\frac{-\pi}{\sin \pi\lambda} a(\lambda) e^{-it\lambda}\, d\lambda.$$
The integrals defined above are independent of $\sigma$ and $f$ extends as a holomorphic function to a neighbourhood  $\{z\in\C\mid |\Re z|<\pi-A\}$ of $i\R$.

\item Also \begin{equation}\nonumber
\int_\R f(ix) e^{i\lambda x}\,dx=-\frac{\pi}{\sin \pi\lambda} a(\lambda),
\end{equation} holds for the extension of $f$ to $i\R$ and for all $\lambda\in\mathbb C$ with $0<\Re\lambda<\delta$.
\end{enumerate}

\end{theorem}

Bertram (in \cite{Bertram}) provides a group theoretical interpretation of the theorem in the following way: Consider $x\mapsto e^{i\lambda x}, \lambda\in\C$ and $x\mapsto e^{i kx}, k\in\Z$ as the spherical functions on $X_G=\R$ and $X_U=U(1)$ respectively. Both $X_G$ and $X_U$ can be realized as the real forms of their complexification $X_\C=\C$. Let $\wtilde{f}$ and $\what{f}$ denote the spherical transformation of $f$ on $X_G$ and on $X_U$ respectively. Then equation (\ref{eqn-11}) becomes, $$\wtilde{f}(\lambda)=-\frac{\pi}{\sin\pi\lambda} a(\lambda), \hspace{.2in} \what{f}(k)=(-1)^k a(k).$$ Using the duality between $X_U=U/K$ and $X_G=G/K$ inside their complexification $X_\C=G_\C/K_\C$, Bertram proved an analogue of Ramanujan's Master theorem for semisimple Riemannian symmetric spaces of rank one. This theorem was further extended to arbitrary rank semisimple Riemannian symmetric spaces by \'{O}lafsson and Pasquale (see \cite{Olafsson-1}). It was also extended for the hypergeometric Fourier transform associated to root systems by \'{O}lafsson and Pasquale (see \cite{Olafsson-2}) and also to the radial sections of line bundles over Poincar\'{e} upper half plane by Pusti and Ray (\cite{Ray-Pusti}) .

In this paper we prove an analogue of this theorem in the setting of Strum-Liouville operator. We consider the eigenfunction $\varphi_\lambda$ of the Sturm-Liouville operator (\ref{defn-phi-lambda}) and think of it as an analogue of spherical function $x\mapsto e^{i\lambda x}$ on $\R$. Next we consider the operator 
\begin{equation}\label{defn-L-tld}
 L=\frac{d^2}{dt^2} +  \frac{\tilde{A}'(t)}{\tilde{A}(t)} \frac{d}{dt},
\end{equation}
on $(0, \frac{\pi}{2})$ where $\tilde{A}(t)=(-i)^{2\alpha +1}A(it)$ and $A$ is given in (\ref{defn-A}). The functions $\Psi_j$'s (defined in section 3) are (countable) eigenfunctions of $-L$ with eigenvalue $\nu_j$. These $\Psi_j$'s are orthonormal basis of $L^2\left((0, \frac{\pi}{2}), \widetilde{A}(t)dt\right)$.  We think these $\Psi_j$'s  as an analogue of $x\mapsto e^{ikx}$ on $\mathrm{U}(1)$. These $\varphi_\lambda$ and $\Psi_j$ are related by 
\begin{equation}
\label{relation-phi-psi}
\Psi_j(t)=c_j\varphi_{i\sqrt{\nu_j + \rho^2}}(it), \text{ on } (0, \frac{\pi}{2}).
\end{equation}
In the non perturbed case i.e. the case when $B=1$ in (\ref{defn-A}), the functions $\Psi_j$'s reduces to Jacobi polynomials. 
Using relation (\ref{relation-phi-psi}) we state and prove an analogue of Ramanujan's Master theorem (see Theroem \ref{Ramanujan-per}) for the Sturm Liouville operator.  Since $\Psi_j$'s are orthonormal basis of $L^2\left((0, \frac{\pi}{2}), \widetilde{A}(t)dt\right)$, we can think of the series (\ref{fourier-series}) as the Fourier series corresponding to the operator $-L$. The main crux of the proof of the theorem is to find a function $b$ for which (\ref{integral-f}) holds. In the Euclidean case and in the non perturbed case the function $b$ is related to the reciprocal of sine function but here in this perturbed case reciprocal of sine function will not work. Instead, here the function $b$ is related to inverse of some sine type function (see (\ref{defn-b}) for exact definition). From (\ref{ab-lambda}) we can also interpolate the values of $a$ to continuous parameter from the discrete parameter.

The plan of the paper is as follows: In section 2 we define the necessary terminology and state some facts with references. Section $3$ and $4$ are devoted to developing the Fourier series analogue of the operator given by (\ref {defn-L-tld}), the relation (\ref{relation-phi-psi}) and the corresponding function $b(\cdot)$ as mentioned above.  After developing all the machinery we prove our main theorem in section 5. For the sake of completeness, in section $7$ we
state some standard theorems in our context and definitions of well known concepts. 

\section{Preliminaries}
Throughout this paper we always assume that, $\alpha, \beta\geq -\frac{1}{2}$.  Let $\Omega=\{t+ is\mid |s|<\frac{\pi}{2}\}$. We define $A:\Omega\rightarrow \C$  by 
\begin{equation}
\label{defn-A}
A(z)=(\sinh z)^{2\alpha +1} (\cosh z)^{2\beta +1} B(z),
\end{equation}
where $B:\Omega\rightarrow \C\setminus \{0\}$ is holomorphic. In this paper we assume the following $(1)$ to $(4)$ conditions:
\begin{enumerate}
\item The function $B$ is even on $\Omega$, positive on $\R$ and $B\left|_{\{is: \, |s|<\frac{\pi}{2}\}}\right.>0$.
\item The function $B$ has an even (with respect to $i\frac{\pi}{2}$) holomorphic extension to a neighboorhood of $i\frac{\pi}{2}$.

\item
The function $\frac{A'(t)}{A(t)}$ is non-negative decreasing function, for large $t.$
We define \begin{equation}\nonumber
2\rho=\lim_{t\rightarrow \infty}\frac{A'(t)}{A(t)}.
\end{equation}
Also assume that (as in \cite{Bensaid}), there exists $\delta>0$ such that for all $t$ in $[t_0, \infty)$ (for some $t_0>0$)
\begin{equation}\label{cond-A}
\frac{A'(t)}{A(t)}=\left\{ \begin{array}{lll}
2\rho + e^{-\delta t} D(t) & \text{ if } \rho>0, \\ \\
 \frac{2\alpha +1}{t} + e^{-\delta t} D(t) & \text{ if } \rho=0,
\end{array} \right.
\end{equation}
where $D$ is a smooth bounded function  such that its derivatives are also bounded. 

\item The function $G$ defined in equation (\ref{defn-G}) is integrable along any straight line in $\Omega$.
\end{enumerate}

The condition (\ref{cond-A}) above assures that for large $t$,
\begin{equation}\label{growth A}
A(t)=\left\{\begin{array}{lll}
O(e^{2\rho t}) & \text{ if } & \rho>0,\\ \\
O(|t|^{\alpha + 1}) & \text{ if } & \rho=0.
\end{array}  \right.
\end{equation}
We consider the following Sturm-Liouville operator
\begin{equation}
\mathcal L=\frac{d^2}{dt^2} +  \frac{A'(t)}{A(t)} \frac{d}{dt}.
\end{equation} 

 For each $\lambda\in\C$, we define $\varphi_\lambda$ as the unique solution of 
\begin{equation}
\label{defn-phi-lambda}
\mathcal L f + (\lambda^2 + \rho^2)f=0, \text{ with } f(0)=1, f'(0)=0.
\end{equation}

For the case when $B(t)=1$ for all $t$, 
the  Sturm-Liouville operator $\mathcal L$ is the radial part of the Laplace-Beltrami operator on the rank one symmetric spaces of noncompact type $X=G/K$ and in this case the function $\varphi_\lambda$ (defined in (\ref{defn-phi-lambda})) becomes the elementary spherical function on $X$. We call the case $B=1$ as the {\em non perturbed} case and otherwise as {\em perturbed case}. We remark that in the non perturbed case all the stated conditions $(1)-(4)$ are satisfied automatically.

We have the following properties of $\varphi_\lambda$ (\cite{Bloom, Brandolini-Gigante}):
\begin{enumerate}
\item For each $t\in \R$, the function $\lambda\mapsto \varphi_\lambda(t)$ is even, entire.
\item For each $\lambda \in \C$, the function $t \mapsto \varphi_\lambda(t)$ is even.
\item For $\lambda\in\C$ with $|\Im\lambda|\leq \rho$,  $|\varphi_\lambda(t)|\leq 1$ for all $t\in\R$.
\item For $\lambda\in \C, t\in\R$, $|\varphi_\lambda(t)|\leq C (1 + |t|) e^{-\rho |t|}e^{|\Im\lambda|\,|t|}$.
\end{enumerate}
For a function $f\in L^1(\R, A(t)\,dt)$, the (Sturm-Liouville) Fourier transform of $f$ is defined by
\begin{equation}
\what{f}(\lambda)=\int_\R f(t) \varphi_\lambda(t) A(t)\, dt.
\end{equation}
Also for suitable function $f$ the inverse Fourier transform is given by
\begin{equation}
f(t)=\frac{1}{4\pi}\int_\R \what{f}(\lambda) \varphi_\lambda(t) |c(\lambda)|^{-2}\, d\lambda,
\end{equation}
where $c(\lambda)$ is the Harish-Chandra $c$-function associated with the Sturm-Liouville operator.
Let $C_{c, R}^\infty(\R)$ be the space of all compactly supported smooth functions on $\R$ with support in $[-R, R]$. Also, let $PW_R(\C)$ be the space of all entire functions $F$ such that for each $N\in\N$,
\begin{equation}
\sup_{\lambda\in\C} (1 + |\lambda|)^N |F(\lambda)| e^{-R|\Im\lambda|} <\infty.
\end{equation}
We also denote $PW_R(\C)_e$ as the space of even functions in $PW_R(\C)$.
Then we have the following Paley-Wiener theorem:
\begin{theorem}$($\cite[Theorem 3]{Chebli}$)$
The (Sturm-Liouville) Fourier transform $f\mapsto \what{f}$ is a topological isomorphism between $C_{c, R}^\infty(\R)$ and $PW_R(\C)_e$.
\end{theorem}

\begin{definition}
For $1\leq p\leq 2$, the $L^p$-Schwartz space $\mathcal C^p(\R)$ is the collection of all $C^\infty$ functions $f$ on $\R$ such that for each $N, m\in \N\cup \{0\}$,
\begin{equation}
\nonumber
\sup_{x\in\R} (1 + |x|)^N \left|\mathcal L^m f(x)\right| e^{\frac{2}{p}\rho |x|}<\infty.
\end{equation}
Using (\ref{growth A}) it follows that $\mathcal C^p(\R) \subseteq L^p(\R, A(t)dt)$.
\end{definition}
For $1\leq p\leq 2$, let $S_p=\{\lambda\in\C\mid |\Im \lambda|\leq (\frac 2p-1)\rho\}$. Also let $\mathcal S_p(\R)_e$ be the collection of all even $C^\infty$ functions on $S_p$ such that for each $N, m\in \N\cup \{0\}$, \begin{equation}
\nonumber
\sup_{\xi\in S_p} (1 + |\xi|)^N \left|\frac{d^m}{d\xi^m} f(\xi)\right|<\infty.
\end{equation}
\begin{theorem} $($\cite{Bensaid}$)$
The map $f\mapsto \what{f}$ is a topological isomorphism between $\mathcal C^p(\R)$ and $\mathcal S_p(\R)_e$.
\end{theorem}

We already know that the function $x\mapsto \varphi_\lambda(x)$ is $C^\infty$ on $\R$. But the following theorem states that the function has a holomorphic extension to the ``crown domain'' $\Omega:=\{z\in\C\mid |\Im z|<\frac{\pi}{2}\}$.
\begin{lemma}\label{holo-extn-phi}
The function $x\mapsto \varphi_\lambda(x)$ has holomorphic extension to $\Omega$.
\end{lemma}

\begin{proof}
We recall that $\varphi_\lambda$ is the unique solution of 
\begin{equation}\label{equation-phi-lambda}
\frac{d^2f}{dt^2} + \frac{A'(t)}{A(t)} \frac{df}{dt} + (\lambda^2 + \rho^2) f=0,
\end{equation} 
with $f(0)=1, f'(0)=0$.
In \cite[Theorem 2]{Fitouhi} it is proved that $\varphi_{\lambda}$ has a real analytic extension on the real line around zero. The same proof also works in our case to show that $\varphi_\lambda$ has a holomorphic extension to a neighboorhood of $0$ in $\C$, call it $\Omega_0$. Let $\Omega_1=\Omega\setminus (-\infty, 0]$, $\Omega_2=\Omega\setminus [0, \infty)$ and let $y_0\in \Omega_1\cap \Omega_0\cap \Omega_2$. Then from Theorem \ref{diffequation-holo} (in Appendix), there exists a unique holomorphic solution $f$ on $\Omega_1$ of (\ref{equation-phi-lambda}) with initial condition $f(y_0)=\varphi_\lambda(y_0), f'(y_0)=\varphi_\lambda'(y_0)$. Similarly there exists a unique holomorphic solution $f$ on $\Omega_2$ of (\ref{equation-phi-lambda}) with initial condition $f(y_0)=\varphi_\lambda(y_0), f'(y_0)=\varphi_\lambda'(y_0)$. Therefore by analytic continuation it follows  that $\varphi_\lambda$ has a holomorphic extension to $\Omega$.
\end{proof}

Before going further we will rewrite $\mathcal{L}$ as a perturbation of the Bessel equation to deduce some more properties of $\varphi_{\lambda}.$
After applying the classical Liouville transformation i.e. $v(t)=\sqrt{{A}(t)}u(t)$, equation (\ref{equation-phi-lambda}) reduces to 
\begin{equation}\label{neweq}
\frac{d^2v}{dt^2} -\frac{\alpha^2-\frac{1}{4}}{t^2}v -G(t)v + \lambda^2  v=0
\end{equation}
 
 where $$G(t)=\frac{1}{4} \left(\frac{A'(t)}{A(t)}\right)^2 +  \frac{1}{2} \left(\frac{A'(t)}{A(t)}\right)'-\rho^2-\frac{\alpha^2-\frac{1}{4}}{t^2}.$$
Let $$G_0 (t):=\left((\alpha + \frac 12)^2-(\alpha + \frac 12)\right)\coth^2 t + \left((\beta + \frac 12)^2-(\beta + \frac 12)\right)\tanh^2 t + (\alpha + \frac 12)$$  
$$+ (\beta + \frac 12) + 2(\alpha + \frac 12)(\beta + \frac 12)-\rho^2-\frac{\alpha^2-\frac{1}{4}}{t^2}. $$
It is easy to check that $\coth^2 t-\frac{1}{t^2}$ is a holomorphic function on $\Omega,$ hence $G_0$ is a holomorphic function on $\Omega.$

A simple computation shows that 
\begin{equation} \label{defn-G}
G(t)= G_0(t)+(\beta+\frac{1}{2})\tanh t\frac{B'(t)}{B(t)} +(\alpha+\frac{1}{2})\coth t\frac{B'(t)}{B(t)} +\frac{1}{4}\left(\frac{B'(t)}{B(t)}\right)^2 +\frac{B''(t)}{2B(t)}.
\end{equation}
The assumptions on $B$ implies that $B'(0)=0,$ which assure that $(\alpha+\frac{1}{2})\coth t\frac{B'(t)}{B(t)}$ is a holomorphic function on $\Omega$  and therefore $G$ is a holomorphic function on $\Omega.$

\begin{theorem}\label{est-phi}
Let $G$ be integrable along any straight line in $\Omega$. Then for $\lambda (\not=0)\in \C$, there exists a polynomial $P$ and a constant $C>0$ such that
\begin{equation}
\nonumber
|\varphi_\lambda(\xi)|\leq C \,|P(\xi)|\, e^{|\Im(\lambda \xi)|},
\end{equation}
for all $\xi\in\Omega$.
\end{theorem}
The proof of this theorem is similar to \cite[Theorem 1.2]{Brandolini-Gigante} but for the sake of completeness we give the proof of the theorem above, in the Appendix.

\begin{remark}
 In the nonperturbed case (that the case when $B(t)=1$ for all $t$), we have 
 \begin{equation}
 \nonumber
 A(t)=(\sinh t)^{2\alpha +1} (\cosh t)^{2\beta +1},
 \end{equation}
and hence 
$$G(t)=\left((\alpha + \frac 12)^2-(\alpha + \frac 12)\right)\coth^2 t + \left((\beta + \frac 12)^2-(\beta + \frac 12)\right)\tanh^2 t + (\alpha + \frac 12)$$  
$$+ (\beta + \frac 12) + 2(\alpha + \frac 12)(\beta + \frac 12)-\rho^2-\frac{\alpha^2-\frac{1}{4}}{t^2}.$$
This shows that $G$ is in $L^1$ in every direction in the domain $\Omega$. Therefore the condition on $G$ is satisfied automatically for the nonperturbed case.
\end{remark}

Let $\Phi_\lambda$ be the second solution of (\ref{equation-phi-lambda}) on $(0, \infty)$. Also it follows that $\Phi_{-\lambda}$ is a solution of (\ref{equation-phi-lambda}) on $(0, \infty)$. The Wronskian of $\Phi_\lambda, \Phi_{-\lambda}$ is given by (\cite[Corollary 1.12]{{Brandolini-Gigante}}) \begin{equation}
\nonumber
\mathcal W(\Phi_\lambda, \Phi_{-\lambda})=-2i\lambda A(t)^{-1}.
\end{equation}
Therefore, for $\lambda(\not=0)$, the soultions $\Phi_\lambda, \Phi_{-\lambda}$ are linearly independent. Hence there is a function $c$ such that for $\lambda\not=0$,
\begin{equation}
\nonumber
\varphi_\lambda=c(\lambda)\Phi_\lambda + c(-\lambda)\Phi_{-\lambda}.
\end{equation}
It follows from \cite[Theorem 2.4]{{Brandolini-Gigante}} that $c(\lambda)$ is holomorphic for $\Im\lambda<0$. We need the following improved theorem.
\begin{theorem}\label{c-holomorphic}
The function $c$ is holomorphic on $\C\setminus\{0, \frac{m}{2}i\mid m\in \N\}$ if $\rho>0$. Also the function $c$ is homorphic on $\C\setminus\{\frac{m}{2}i\mid m\in \N\}$ if $\rho=0$.
\end{theorem}
\begin{proof} 
For $\lambda(\not=0)\in\C$, we have for $t>0$
\begin{equation}
\nonumber
\varphi_\lambda(t)=c(\lambda)\Phi_\lambda(t) + c(-\lambda)\Phi_{-\lambda}(t).
\end{equation}
Therefore, for $t>0$
\begin{equation}
\nonumber
\varphi_\lambda'(t)=c(\lambda)\Phi_\lambda'(t) + c(-\lambda)\Phi_{-\lambda}'(t).
\end{equation}
Hence, 
\begin{equation}
\label{expression-c}
c(\lambda)=\frac{\varphi_\lambda(t) \Phi_{-\lambda}'(t)-\varphi_\lambda'(t) \Phi_{-\lambda}(t)}{\mathcal W(\Phi_\lambda, \Phi_{-\lambda})(t)}=\frac{i A(t)}{2\lambda}\left(\varphi_\lambda(t) \Phi_{-\lambda}'(t)-\varphi_\lambda'(t) \Phi_{-\lambda}(t)\right).
\end{equation} Now we consider the case when $\rho>0$.
To prove the theorem (in this case) we shall prove that the functions 
\begin{equation}
\nonumber
\lambda\mapsto \Phi_{-\lambda}(t), \Phi_{-\lambda}'(t),
\end{equation}
both are holomorphic on $\C\setminus\{\frac{m}{2}i\mid m\in \N\}$. We have $\Phi_\lambda$ is a solution of 
\begin{equation}\label{eqn-Phi-lambda}
y'' + \frac{A'(t)}{A(t)} y' + (\lambda^2 + \rho^2)y=0.
\end{equation}
Let \begin{equation}
\nonumber 
A(t)=(\sinh t)^{2\alpha +1} (\cosh t)^{2\beta +1} B(t).
\end{equation} Then 
\begin{equation}
\nonumber
\frac{A'(t)}{A(t)}= (2\alpha + 1) \coth t + (2\beta + 1)\tanh t + \frac{B'(t)}{B(t)}.
\end{equation}
Let \begin{equation}
\nonumber
\frac{B'(t)}{B(t)}=\sum_{j=0}^\infty 2 b_j e^{-j t}.
\end{equation}
Then $\rho=\alpha + \beta + 1 + b_0$.
Let \begin{equation}
\nonumber
y(t)=\sum_{n=0}^\infty \Gamma_n(\lambda) e^{(i\lambda-\rho-n)t}
\end{equation}
be a solution of (\ref{eqn-Phi-lambda}). Then putting this in the equation (\ref{eqn-Phi-lambda}) and comparing coefficients we get 
\begin{equation}
\nonumber
\Gamma_1(\lambda)=\frac{-2b_1(i\lambda-\rho)}{1-2i\lambda} \Gamma_0(\lambda)
\end{equation}
\begin{equation}\nonumber
\Gamma_2(\lambda)=\frac{1}{4(1-i\lambda)}\left(-2(2\alpha-2\beta + b_2)(i\lambda-\rho)\Gamma_0(\lambda) -2b_1(i\lambda-\rho-1) \Gamma_1(\lambda)\right),
\end{equation}
and 
\begin{equation}
\nonumber
3(3-2i\lambda)\Gamma_3(\lambda)=-2(2\alpha-2\beta + b_2)(i\lambda-\rho-1)\Gamma_1(\lambda)-2b_3(i\lambda-\rho)\Gamma_0(\lambda)-2b_1(i\lambda-\rho-2)\Gamma_2(\lambda),
\end{equation}
and continued in this way. This shows that each $\Gamma_n$ is holomorphic on $\C\setminus \{-i\frac{m}{2}\mid m\in \N\}$.
As in the classical case of symmetric spaces (see \cite[Ch. IV, Lemma 5.3]{Helgason}), it is easy to check that for any $t_0>0$, 
\begin{equation}\label{est-gamma-n}
|\Gamma_n(\lambda)|\leq K_{\lambda, t_0} e^{n t_0}.
\end{equation}
From equation (\ref{est-gamma-n}) it follows that 
\begin{equation}
\label{defn-Phi}
\Phi_\lambda(t)=e^{(i\lambda-\rho)t}\sum_{n=0}^\infty \Gamma_n(\lambda) e^{-n t}, t>0
\end{equation}
is well defined and converges absolutely and uniformly for $t>0$ and for $\lambda\in \C\setminus\{-i\frac{m}{2}\mid m\in \N\}$.
This shows that $\lambda\mapsto \Phi_\lambda(t), \Phi_\lambda'(t)$ are holomorphic on $\C\setminus\{-i\frac{m}{2}\mid m\in \N\}$.

Now if $\rho=0$, it follows from the expression above that $\Gamma_n(0)=0$ for $n=1, 2, \cdots $. Therefore we have \begin{equation}\nonumber
\Phi_0(t)=\Gamma_0(0) \text{ for all } t>0.
\end{equation}
 Then the function \begin{equation}\nonumber 
\lambda\mapsto\varphi_\lambda(t) \Phi_{-\lambda}'(t)-\varphi_\lambda'(t) \Phi_{-\lambda}(t)
\end{equation} at the point $\lambda=0$ reduces to $-\Gamma_0(0) \varphi_0'(t)$ which is equal to $0$, as in this ($\rho=0$) case, $\varphi_0'(t)=0$. Hence from (\ref{expression-c}) it follows
that the function $c$ is holomorphic at $\lambda=0$. Rest is similar to above case. This completes the proof. 
\end{proof}

We have the following corollary:
\begin{corollary}\label{estimate-c}
\begin{enumerate}
\item The $c$-function has a simple pole at $\lambda=0$ if $\rho>0$ but it is holomorphic at $\lambda=0$ if $\rho=0$.
\item Let $\delta<\frac 12$. Then for $|\Im \lambda|\leq \delta$ and for all $\lambda$ outside a neighbourhood of $0$
\begin{equation}
\nonumber
|c(\lambda)|\leq C p(|\lambda|),
\end{equation} 
for some polynomial $p$. 
\end{enumerate}
\end{corollary}
\begin{proof}
First part follows from (\ref{expression-c}). Second part  follows from the series expansion (\ref{defn-Phi}) of $\Phi_\lambda$ and noting that each $\Gamma_n$ is a rational function.
\end{proof}

\section{Compact Case} 
In order to obtain an analogue of Ramanujan's Master theorem, Theorem \ref{eigenexpansion} for the Sturm Liouville operator, we need to develop the corresponding Fourier series for $\mathcal{L}$ when restricted to, $z=it, t\in (0,\pi/2).$
In this section we will define a positive, symmetric densely defined differential operator $-L$ on $(0,\pi/2)$ on a suitable Hilbert space such that the following holds \begin{equation}\label{relation}-L w(t)=\mathcal{L}u(z)_{z=it},\end{equation} where $w(t)=u(it)$ for $u$ twice differentiable defined in $\Omega\setminus \{0\}$.  We will  study the spectral decomposition of $-L$ on $(0,\pi/2)$ under suitable boundary conditions such that we obtain the eigenfunctions of $L$ as the restriction of the eigenfunctions of $\mathcal{L}$ on $z=it$. This will be in direct analogy of the functions $\{e^{i\lambda z}\}$ as discussed in the introduction.
Let us recall that $$A(t)=(\sinh t)^{2\alpha+1} (\cosh t)^{2\beta+1} B(t),$$ where $B$ is a non zero, even holomorphic function on $\Omega.$
We also have $B(it)>0$ when $t\in(-\pi/2,\pi/2).$
We define \begin{equation}\nonumber
\tilde{A}(t)=(-i)^{2\alpha + 1} A(it).
\end{equation}
Let $\tilde{B}(t)=B(it)$.
Indeed $$\tilde{A}(t)=(\sin t)^{2\alpha+1} (\cos t)^{2\beta+1} \tilde{B}(t).$$
Clearly $\tilde{A}>0$ on $(0,\pi/2).$
We define the Sturm Liouville operator corresponding to $\tilde{A}$ on $(0,\pi/2)$ as
\begin{equation}
 L=\frac{d^2}{dt^2} +  \frac{\tilde{A}'(t)}{\tilde{A}(t)} \frac{d}{dt}.
\end{equation}
It is easy to verify that this choice of $L$ satifies (\ref{relation}) above.
%Integration by parts shows that $- \tilde{L}$ is a positive, symmetric operator on \newline$f\in \mathcal{D}_0:=\{f\in \mathcal{D}: f ~~ \text{is compactly supported on}~~ (0,\pi/2)\}.$ 
We define 
$\mathcal{D}(L)$ to be the space of $f\in L^2\left(\left(  0,\pi/2\right),\tilde{A}(t) dt \right)$ such that  $f$ and $\tilde{A}(\cdot)f^{\prime}$ are absolutely continuous on
any compact subinterval of $\left(  0,\pi/2\right)$ and  $L\left(
f\right)  \in L^{2}\left( (0,\pi/2),\tilde{A}(t) dt\right)$.

The operator $-L$ is a densely defined operator from $\mathcal{D}(L)$ to $L^2\left((0,\pi/2), \tilde{A}(t)dt\right).$ Let \begin{equation} \nonumber 
\mathcal{D}(L)_0:=\{f\in \mathcal{D}(L): f ~~ \text{is compactly supported on}~~ (0,\pi/2)\}.
\end{equation}
Let us denote $-L$ restricted on $\mathcal{D}(L)_0$ as $-L_0.$
An integration by parts argument shows that $-L_0$ is a positive and symmetric operator on $\mathcal{D}(L)_0.$ We need to obtain a self adjoint extension of $-L_0$ on $L^2((0,\pi/2), \tilde{A}(t)dt)$ with suitable boundary conditions so that the eigenfunctions are restriction of $\varphi_{\lambda}$ to $z=it, t\in(0,\pi/2)$ for $\lambda$ related to the spectrum of the self adjoint extension of $-L_0.$ 

Let $u$ be an eigenfunction of $-L$ with eigenvalue $\mu.$
After applying the classical Liouville transformation i.e. $v(t)=\sqrt{\tilde{A}(t)}u(t)$ we get another differential operator $-l$ such that $v$ is an eigenfunction of $-l$ with eigenvalue $\mu.$ One can explicitly write the expression of $-l$ as follows: $$-l=-\frac{d^2}{dt^2}+q(t),$$
where $$q(t)=\frac{1}{4}\left(\frac{\tilde{A}'(t)}{\tilde{A}(t)}\right)^2 + \frac{1}{2}\left(\frac{\tilde{A}'(t)}{\tilde{A}(t)}\right)'.$$
In fact  \begin{equation}\label{q(t)} q(t)= \left((\alpha^2-\frac{1}{4}) \cot^2 t +(\beta^2-\frac{1}{4})\tan^2 t -\chi(t)\right),\end{equation} 
where \newline $$\chi(t)=\left(\beta+\frac{1}{2}\right)\frac{\tilde{B}'(t)}{\tilde{B}(t)}\tan t - \left(\alpha+\frac{1}{2}\right)\frac{\tilde{B}'(t)}{\tilde{B}(t)}\cot t + \frac{1}{4}\left(\frac{\tilde{B}'(t)}{\tilde{B}(t)}\right)^2 -\frac{1}{2}\frac{\tilde{B}''(t)}{\tilde{B}(t)}.$$
It follows  from the assumptions on  $\tilde{B}$ that $\chi$ is a smooth function on $[-\frac{\pi}{2},\frac{\pi}{2}].$
 Therefore $q$ has singularities only at $0$ and $\pi/2$. A simple evaluation gives that 
  \begin{equation}\label{sturmtrans}\sqrt{\tilde{A}(t)}Lu(t)=lv(t),\end{equation} 
 where $v(t)=\sqrt{\tilde{A}(t)}u(t).$
 The unbounded operator $l:\mathcal{D}(l)\subset L^{2}(0,\pi/2)\rightarrow L^{2}(0,\pi/2)$ is
defined on
\[
\mathcal{D}(l)=\left\{  f\text{ and }f^{\prime}\text{ AC on
any compact subinterval of }\left(  0,\pi/2\right)  \text{, }f,l\left(
f\right)  \in L^{2}\left(  0,\pi/2\right)  \right\}  .
\]
Here AC stands for absolutely continuous.
We observe that $\tilde{A}>0$ and $\sqrt{\tilde{A}}$ is an AC function on any compact sub interval of $(0,\pi/2)$ and therefore it follows that \begin{equation}\label{side}\mathcal{D}(l)=\sqrt{\tilde{A}(\cdot)}\mathcal{D}(L).
\end{equation} 
Let 
\begin{equation}\nonumber
\mathcal{D}(l)_0:=\{f\in \mathcal{D}(l): f ~~ \text{is compactly supported on}~~ (0,\pi/2)\}.\end{equation}  A similar identity 
 (\ref{side}) also holds for $\mathcal{D}(l)_0$ and $\mathcal{D}(L)_0.$ We denote $-l$ restricted on $\mathcal{D}(l)_0$ as $-l_0.$
Let $v_1,v_2\in \mathcal{D}(l)_0.$
The following holds:
$$\left(-L_0u_1,u_2\right)_{L^2((0,\pi/2),\tilde{A}(t)dt)}=\left(-l_0v_1,v_2\right)_{L^2(0,\pi/2)},$$
where $v_i(t)=\sqrt{\tilde{A}(t)}u_i(t) $ for $i=1,2.$
%It is easy to see that $\sqrt{\tilde{A}(t)}\tilde{L}u(t)=\tilde{l}v(t).$ 
The last identity just uses the relation between $u_i$ and $v_i, i=1,2$ and equation (\ref{sturmtrans}). This shows that $-l_0$ is a positive and symmetric operator on $\mathcal{D}(l)_0.$
In fact if $u$ is an eigenfunction of $-L$ with eigenvalue $\mu,$ the equation (\ref{sturmtrans}) gives that $\sqrt{A(t)}u(t)$ is an eigenfunction of $-l$ with the same eigenvalue and vice versa too.
Therefore the eigenfunctions of $ -L$ and $-l$ are in one to one correspondence by Liouville transformation.
A simple computation gives \begin{equation}\nonumber
\|v\|_{L^2(0,\pi/2)}=\|u\|_{L^2((0,\pi/2),\tilde{A}(t)dt)}.
\end{equation}
In fact the map $u\rightarrow \sqrt{A(\cdot)}u(.)$ is an isometry from $L^2((0,\pi/2),\tilde{A}(t)dt))$ onto $L^2(0,\pi/2).$ In view of the above relation between $-L$ and $-l$, in order to obtain self adjoint extension of $-L_0$ on $L^2((0,\pi/2),\tilde{A}(t)dt),$ it is enough to obtain a self adjoint extension of $-l_0$ on $L^2(0,\pi/2).$

We have the following Theorem:
\begin{theorem}\label{spectral}
The operator $-L_{0}$  has a self adjoint extension (with abuse of notation call it $-L$) on $\tilde{D}:=\left(\sqrt{\tilde{A}}\right)^{-1}\mathcal{D}$ $($where $\mathcal{D}\subset{L^2(0,\pi/2)}$ is as defined in Proposition \ref{boundary conditions}$)$. 
The spectrum of $-L$ is purely discrete, bounded below and all
eigenvalues are simple. The eigenvalues can be ordered by%
\[
\nu_0<\nu_{1}<\ldots<\nu_{n}<\ldots
\]
with $\nu_{n}\rightarrow+\infty$. The eigenfunctions $\{\Psi_j\}$ corresponding to the eigenvalues $\{\nu_j\}$ form an orthonormal basis of $L^2((0,\pi/2),\tilde{A}(t)dt)$ and there exist constants $c_j$ with a polynomial growth in $j$ such that $$\Psi_j(t)=c_{j} \varphi_{i\sqrt{\nu_j +\rho^2}}\left(  it\right) $$ for all $j\geq j_0,$ where $j_0$ is the least natural number for which $\nu_j+\rho^2 >0.$
When $\nu_j+\rho^2\leq 0,$ 
let $\beta_j:=\sqrt{-\left(\nu_j +\rho^2\right)}$. Then for $j=0,1,2,...j_0-1$ we have $$\Psi_j(t)=c_{j} \varphi_{\beta_j}\left(  it\right).$$
When $\alpha,\beta>0,$ the operator $-L$ is  the Friedrich's extension of $-L_{0}$ and $\nu_j's$ are non negative.	
\end{theorem}

The proof of the above theorem is given at the end of this section.

\noindent {\bf Relation between $\mathcal{L}$ and $-L$:}
Let $u$ be a twice differentiable function on $\Omega$. Define $w(t):=u(it).$ Using the relation between $\tilde{A}$ and $A$ it is easy to see that the relation (\ref{relation}) holds, more precisely $$\mathcal{L}u(z)_{z=it}=-Lw(t)$$
We define $w_{\mu}(t)=\varphi_{i\sqrt{\mu+\rho^2}}(it),$ where $\mu+\rho^2\in \C\setminus (-\infty,0].$ 
We know that \begin{equation} \nonumber \mathcal{L}\varphi_{i\sqrt{\mu+\rho^2}}(z)=\mu \varphi_{i\sqrt{\mu+\rho^2}}(z), \end{equation} for all $z\in\Omega\setminus \{0\}$ (in particular when $z=it, t\in(0,\pi/2)$). Therefore we have \begin{equation} \nonumber - Lw_{\mu} (t) = \mu w_{\mu}(t), t\in (0,\pi/2).\end{equation}
In the case when $\mu+\rho^2\leq 0,$ define $\beta=\sqrt{-(\mu+\rho^2)}.$
By the same principle as above we can check that $-Lu_{\beta}(t)=\beta u_{\beta}(t),$ where $u_{\beta}(t)=\varphi_{\beta}(it).$
We define $$\tilde{\varphi}_{i\sqrt{\mu+\rho^2}}(t)=\sqrt{\tilde{A}(t)}\varphi_{i\sqrt{\mu+\rho^2}}(it)$$ for $t\in(0,\pi/2).$ By the correspondence between $ -L$ and $-l$ (\ref{sturmtrans}), we also have 
\begin{equation}\nonumber 
-l\tilde{\varphi}_{i\sqrt{\mu+\rho^2}}(t)= \mu \tilde{\varphi}_{i\sqrt{\mu+\rho^2}}(t), t\in(0,\pi/2). \end{equation}
Similarly when $\mu+\rho^2\leq 0,$ define $\beta=\sqrt{-\left(\mu+\rho^2\right)}.$ We get that $-L\tilde{\varphi_{\beta}}=\mu \tilde{\varphi_{\beta}}.$

%The unbounded operator $\tilde{l}:\mathcal{D\subset}L^{2}(0,\pi/2)\rightarrow L^{2}(0,\pi/2)$ is
%defined in
%\[
%\mathcal{D=}\left\{  f\text{ and }f^{\prime}\text{ absolutely continuous on
%any compact subinterval of }\left(  0,\pi/2\right)  \text{, }f,\ell\left(
%f\right)  \in L^{2}\left(  0,\pi/2\right)  \right\}  .
%\]
\noindent{\bf Spectral Decomposition of $-l$:}
Let us recall that $$-l=-\frac{d^2}{dt^2}+q(t),$$
where $q(t)$ is defined as in equation (\ref{q(t)}).
 It is known that $0$ and $\pi/2$ are non-oscillatory end points of $-l,$ i.e. there exist solutions of $-lu_i =\mu u_i, i=1,2$ such that $u_i$ is non zero in $(0,\epsilon)$ and $(\pi/2-\epsilon,\pi/2)$ respectively for $i=1,2$ (see section $2$ \cite{Yurko}) where $u_i's$ are defined on $(0,\pi/2)$. (See Appendix for further details.)
 
 Let us fix $\mu\geq 0$. We know that $-l\tilde{\varphi}_{i\sqrt{\mu+\rho^2}}(t)=\mu \tilde{\varphi}_{i\sqrt{\mu+\rho^2}}(t), t\in (0,\pi/2).$
 As stated in the last section, $\varphi_{i\sqrt{\mu+\rho^2}}(0)=1, \varphi'_{i\sqrt{\mu+\rho^2}}(0)=0$ and $\sqrt{\tilde{A}(t)}\sim t^{\alpha +1/2},\sqrt{\tilde{A}'(t)}\sim t^{\alpha -1/2}$ near $0+$. Therefore $\tilde{\varphi}_{i\sqrt{\mu+\rho^2}}(t)\sim t^{\alpha +1/2}$ and $\tilde{\varphi}'_{i\sqrt{\mu+\rho^2}}(t)\sim t^{\alpha -1/2}$ in some right neighbourhood of $0.$ 
 In fact one can also construct another solution of 
  $-lu=\mu u$ in $(0,\pi/2)$
 such that it satisfies $u(t)\sim t^{-\alpha +1/2}$ and and $u'(t)\sim t^{-\alpha-1/2}$ and it is also linearly independent to $\tilde{\varphi}_{i\sqrt{\mu+\rho^2}}$ (see section $2$ \cite{Yurko}). Similarly, given $\tilde{\mu}$, we can also find two linearly independent eigenfunctions $W_{\widetilde{\mu},\pm\beta}$ defined on $(0,\frac{\pi}{2})$ satisfying $$-lW_{\widetilde{\mu},\pm\beta}(t)=\tilde{\mu}W_{\widetilde{\mu},\pm\beta}(t), t\in (0,\frac{\pi}{2})$$  such that $W_{\widetilde{\mu},\pm\beta}(t)\sim (\frac{\pi}{2}-t)^{\pm\beta+\frac{1}{2}}$ and $W_{\widetilde{\mu},\pm\beta}'(t)\sim (\frac{\pi}{2}-t)^{\pm\beta-\frac{1}{2}}$ (see section $2$ \cite{Yurko}).
For the $\beta\leq \alpha$ we have the following proposition. The other case i.e. $\alpha<\beta$ follows similarly.
By \cite[Theorem 4.2]{Niessen} \cite[Theorem 50, page 1478]{DS} and \cite[Proposition 9]{JG}, the following holds:  \begin{proposition}
\label{boundary conditions}The operator $-l_0$ has a self adjoint extension (call it $-l$) on %
\[
\mathcal{D}=\left\{
\begin{array}
[c]{ll}%
\left\{  y\in\mathcal{D}(l):\left[  y,\tilde{\varphi}_{i\sqrt{\mu+\rho^2}}\right]  _{l}\left(
0\right)  =\left[  y,W_{\widetilde{\mu},\beta}\right]  _{l}\left(
\pi/2\right)  =0\right\}  & \text{for }-1/2<\beta\leq\alpha<1\\
\left\{  y\in\mathcal{D}(l):\left[  y,W_{\mu,\beta}\right]  _{l}\left(
\pi/2\right)  =0\right\}  & \text{for }-1/2<\beta<1\leq\alpha\\
\mathcal{D}(l)& \text{for }1\leq\beta\leq\alpha.
\end{array}
\right.
\]  Here%
\begin{align*}
\left[  y,u\right]  _{l}\left(  0\right)   &  =\lim_{t\rightarrow0+}\left(
y\left(  t\right)  \overline{u^{\prime}\left(  t\right)  }-y^{\prime}\left(
t\right)  \overline{u\left(  t\right)  }\right) \\
\left[  y,u\right]  _{l}\left(  \pi/2\right)   &  =\lim_{t\rightarrow
\pi/2-}\left(  y\left(  t\right)  \overline{u^{\prime}\left(  t\right)
}-y^{\prime}\left(  t\right)  \overline{u\left(  t\right)  }\right).
\end{align*}
The operator $-l$ is bounded from below in $L^2(0,\pi/2)$ and the domain is independent of the choice of $\mu$ and $\widetilde{\mu}$.
The spectrum of $-l$ is purely discrete, bounded below and all
eigenvalues are simple. The eigenvalues can be ordered by%
\[
\nu_{0}<\nu_{1}<\ldots<\nu_{n}<\ldots
\]
with $\nu_{n}\rightarrow+\infty$.
\end{proposition}

%{\bf When $\alpha$, $\beta>0$}
Now we are in a position to prove  Theorem $\ref{spectral}$.
  \begin{proof}[Proof of Theorem $\ref{spectral}$]
The existence of a self adjoint extension of $-L_0$ on $\tilde{D}$ follows from the equation (\ref{sturmtrans}) and the relation between $\mathcal{D}(L)_0$ and $\mathcal{D}(l)_0$. In fact the eigenvalues of $-L$ are same as that of $-l$ as explained earlier and the eigenfunctions are related by the classical Liouville transformation.

 When $\alpha,\beta >0$, it is clear from the above estimates that $\tilde{\varphi}_{i\sqrt{\mu+\rho^2}}$ and $W_{\widetilde{\mu},\beta}$ are principal solutions respectively at $0$ and $\pi/2$ (See Appendix). In fact the boundary conditions in the above proposition correspond to that of with respect to the principal solution as in \cite[Theorem 4.2]{Niessen}. By uniqueness of the principal solutions upto constant multiples the boundary conditions coincide with that of the Friedrich's extension. 
 
 Therefore when $\alpha,\beta>0$ we obtain the Friedrich's extension of $-L$ on $\tilde{D}.$   The lower bound of the Friedrich's extension is same as that of $-L_0$. 
   Therefore the self adjoint extension considered above is also non negative when $\alpha,\beta>0.$
This implies that $\nu_j$'s are non negative. 
Let $\{u_j\}_{j\geq 0}$ be the eigenfunctions of $-l$ such that $-lu_j=\nu_j u_j.$
We further assume that $\|u_j\|_{L^2(0.\pi/2)}=1$ for all $j\geq 0.$ 
It is clear that $\{u_j\}_{j\geq 0}$ form an orthonormal basis of $L^2(0,\pi/2).$
By Proposition \ref{boundary conditions}, it follows that there exist
constants $c_{j}$ and $d_{j}$ such that for all $t\in\left(  0,\pi/2\right)
$,
\[
u_{j}\left(  t\right)  =c_{j}\tilde{\varphi}_{i\sqrt{\nu_j +\rho^2}}\left(  t\right)  =d_{j}%
W_{\nu_{j},\beta}\left(  t\right)  .
\]
This is clearly true for $\alpha<1$ due to the boundary conditions of
Proposition \ref{boundary conditions}, and in the other cases it follows from
the uniqueness of $\tilde{\varphi}_{i\sqrt{\nu_j +\rho^2}}$ among the solutions of the equation
$-lu=\nu_{j}u$ which are in $L^{2}\left(  \left(  0,\varepsilon\right)
\right)  $, and of $W_{\nu_{j},\beta}$ in $L^{2}\left(  \pi/2-\varepsilon
,\pi/2\right)  .$
The constant $c_j$ has polynomial growth in $j$  (see \cite{JG} Lemma $13$, Page $21$).

We define $\Psi_j(t)=\left(\sqrt{\tilde{A}(t)}\right)^{-1} u_j(t), t\in (0,\pi/2).$ Clearly $\{\Psi_j\}_{j\geq 1}$ forms an orthonormal basis of $L^2((0,\pi/2), \tilde{A}(t)dt)$ satisfying 
$-L\Psi_j(t)=\nu_j \Psi_j(t), t\in(0,\pi/2)$.

Therefore we have \begin{equation}\nonumber
\Psi_j(t)=c_{j} \varphi_{i\sqrt{\nu_j +\rho^2}}\left(  it\right). 
\end{equation}

\end{proof}
%Let the discrete eigenvalue of the Laplacian (in the compact case) be $\lambda_1, \lambda_2, \cdots$ with eigenfunctions $\Psi_{\lambda_1}, \Psi_{\lambda_2}, \cdots$. Then we have \begin{equation}\label{relation-phi-psi}
%\Psi_{\lambda_n}(t)=c_n \varphi_{i(\lambda_n +\rho)}(t).
%\end{equation}

\begin{corollary}\label{est-psi-cor} For each $m\geq 0$, the function $\Psi_m$ satisfies the following inequality:
\begin{equation}\label{estimate-psi}
\left| \Psi_{m}(t + is)\right| \leq C_m Q(|\sqrt{\nu_m +\rho^2}|)\, |P(t + is)| \,e^{|s|\sqrt{\nu_m +\rho^2}},
\end{equation}
for some polynomials $P$ and $Q$.
\end{corollary}
\begin{proof}
It follows from Theorem \ref{est-phi}, that \begin{equation}
\nonumber
\left|\varphi_{\lambda_1 + i\lambda_2}(t + is)\right|\leq C Q(|\lambda|) |P(t + is)| e^{|\lambda_1 s + \lambda_2 t|}.
\end{equation}
Hence the required inequality will follow from Theorem \ref{spectral}.
\end{proof}

\section{Sine type function}
In this section we consider the case when $\alpha, \beta>0$. We recall that $\nu_n$'s  ($n=0, 1, 2, \cdots $) are eigenvalues of $-L$ with eigenfunctions $\Psi_n$ respectively. In this case (i.e. for $\alpha, \beta >0$), $\nu_0\geq 0$ and $\nu_n>0$ for $n=1, 2, \cdots $. We define the following function:
\begin{equation}
S(z)= \left\{\begin{array}{lll} \pi z \prod_{n=0}^\infty \left(1 + \frac{z^2}{\nu_n+\rho^2}\right) &\text { if } & \nu_0+\rho^2>0,\\
\pi z^3 \prod_{n=1}^\infty \left(1 + \frac{z^2}{\nu_n+\rho^2}\right) &\text { if } & \nu_0+\rho^2=0. \end{array}\right.
\end{equation} 
The asymptotic expansion of $\nu_n's$ is known. More precisely (see \cite{JG} and Theorem $2_{I}$\cite{G}) \begin{equation}\label{asymptote}
	\sqrt{\nu_n}=2n+1+\alpha+\beta -\frac{\Theta}{4n} + O(n^{-2}),\end{equation} where $\Theta= \alpha^2 + \beta^2 -1/2 +\frac{2}{\pi}\int_0^{\frac{\pi}{2}}\chi(t)dt$ for $n\rightarrow\infty.$
The function $S(z)$ has zeros exactly at $\{\pm i\sqrt{\nu_n+\rho^2}\}_{n\geq 0}$ and $0.$ Clearly $\sqrt{\nu_n+\rho^2}$ is also a perturbation of $2n+1+\alpha+\beta$ for $n$ large enough. The function $S(iz)$ is a function of sine type (see  \cite{Levin}).

In order to obtain the main Theorem we need to find a uniform bound on  the residue of $S(z)^{-1}$ at $\{ i\sqrt{\nu_n+\rho^2}\}_{n\geq 0}.$
We prove the following theorem:
\begin{theorem}
Let \begin{equation}
\label{defn-S}
 S_1(z)=\frac{z^2}{S(z)}.
\end{equation}
Let $d_n$ be the residue of $S_1$ at $i\sqrt{\nu_n+\rho^2}$. Then the following holds:
\begin{enumerate}
\item \begin{equation}\label{estimate-S}
|S_1(z)|\asymp  \left\{\begin{array}{lll}|z|^2 e^{-\frac{\pi}{2} |\Re z|}
  &\text { if } & \nu_0+\rho^2>0,\\
 e^{-\frac{\pi}{2} |\Re z|}
   &\text { if } & \nu_0+\rho^2=0. \end{array}\right.\end{equation}
\item \begin{equation}
\label{estimate-residue}
|d_n|\leq c |\nu_n + \rho^2| \text{ for all } n ~~ \text{large}.
\end{equation}
\end{enumerate}
\end{theorem}
\begin{proof} 
Let us put $\rho_0=1+\alpha+\beta.$ By using equation (\ref{asymptote}) we can write \begin{equation} \nonumber \sqrt{\nu_n+\rho^2}=2n+\rho_0 + \frac{2\rho^2-\Theta}{4n} + O(n^{-2}), \end{equation} for $n\geq M,$ $M$ large enough. We first deal with the case $\nu_0+\rho^2>0.$ Let $\mu_n:= \sqrt{\nu_n+\rho^2}.$    We define \begin{equation} \nonumber 
M(z)=\frac{S(iz)}{\pi z} .
\end{equation}
 Then $M$ has zeros of order one precisely on the set  $\{\pm \mu_n\}_{n\geq 0}.$
The residue of $S_1$ at $i\mu_n$ is given by \begin{equation}\nonumber 
d_n:=\lim_{z\rightarrow i\mu_n}(z-i\mu_n)S_1(z).
\end{equation} Since $S_1(z)=z^2(-i\pi z M(-iz))^{-1}$, it is easy to see that $$d_n= -\mu_n^2\lim_{-iz\rightarrow \mu_n}\frac{(-iz-\mu_n)}{-\pi zM(-iz)}.$$ As $\mu_n$ is a zero of $M(z),$ $$d_n=\frac{\mu_n^2}{i\pi \mu_n M'(\mu_n)}.$$
In order to show $d_n\times\mu_n^{-2}$ is bounded for large $n,$ it is enough to show that $(zM'(z))_{ z=\mu_n}$ is bounded below as $n\rightarrow\infty.$
We define \begin{equation}\nonumber 
M_0(z)= \frac{2\sin (\pi(z-\rho_0)/2)}{\pi (z-\rho_0)}.
\end{equation} The function $M_0$ is a sine type function and exponential of type $\pi/2.$
Clearly the zeros of $M$ are the perturbation of the zeros of $M_0$. Indeed from (\ref{asymptote}) it is clear that the perturbation is of order 2. 
Using the notation of \cite[p.86]{Levin}, let $\lambda_k=2k+\rho_0$ and $\psi_k=-\frac{2\rho^2-\Theta}{4k} + O(k^{-2}).$
 By   \cite[Theorem 2, p.86]{Levin} with $n=2$ we have the asymptotic expansion of $M$ in terms of $M_0$ 
\begin{equation}\label{sineasymptote}M(z)= c_0 M_0(z) + z^{-1}(c_1 M_0(z)+ c_2 M_0'(z))+ z^{-2}(d_0 M_0(z)+ d_1 M_0'(z)+ M_0''(z)) + z^{-2}f_2(z),\end{equation}
where $f_2$ is a holomorphic function of exponential type $\leq\frac{\pi}{2}$ and $c_i, d_i, i=0,1,2$ are constants. For getting lower bound on $zM'(z)$ at $\mu_n's$ for $n$ large we use the above expression of $M$ in terms of $M_0.$
On calculating we get that \begin{equation}\nonumber M_0'(z)= \frac{\cos(\pi(z-\rho_0/2))}{z-\rho_0}-\frac{2\sin (\pi(z-\rho_0)/2}{\pi(z-\rho_0)^2}.\end{equation}
Using the fact that $\mu_n$ is a very small perturbation of $2n+\rho_0,$ it is clear that $(zM_0'(z))_{ z=\mu_n}$ is bounded from below for $n$ large enough.
On computing the derivative of $M$ from equation (\ref{sineasymptote}) it is easy to see that the most dominating term of $zM'(z)$ at $z=\mu_n,n\geq M$ is $\mu_nM_0'(\mu_n)$ and the rest of the terms are of the order $\mu_n^{-1}$ which can be made as small as possible for large M. More precisely $$|\mu_n M'(\mu_n)|\asymp |\mu_nM_0'(\mu_n)| + O(\mu_n ^{-1})$$ As $\mu_n M_0'(\mu_n)$ is bounded below for $n\rightarrow\infty,$ it implies that  $(zM'(z))_{ z=\mu_n}$ is bounded from below as $n\rightarrow\infty.$
The pointwise estimate (\ref{estimate-S}) of $S_1$ is clear from the asymptotic expansion (\ref{sineasymptote}) of $M$ above.
When $\nu_0+\rho^2=0$ we define $M(z)=\frac{S(iz)}{z^3}$ and then follow the above proof in a similar manner to get the desired results.
\end{proof}
\section{Main theorem}
In this section we prove Ramanujan's Master theorem for the case when $\alpha, \beta>0$. Let $A, p, \delta$ be real numbers such that $A<\pi/2, \,p>0$ and $0<\delta\leq 1$. Let \begin{equation}\nonumber
\mathcal H(\delta)=\{\lambda\in\C\mid \Im\lambda>-\delta\},
\end{equation} and 
\begin{equation}\nonumber
\mathcal H(A, p, \delta)=\left\{a:\mathcal H(\delta)\rightarrow\C \text{ holomorphic }\mid |a(\lambda)|\leq C \,e^{-p(\Im\lambda) + A|\Re\lambda|} \text{ for all }\lambda\in\mathcal H(\delta)\right\}.
\end{equation}
We recall that the function $S_1$ (defined in (\ref{defn-S})) has simple poles at $\lambda= i\sqrt{\nu_m + \rho^2}, m=0,1, 2, \cdots$ where $\nu_0\geq 0$ and $\nu_m> 0$ for all $m=1, 2, \cdots$. We also recall that  $d_m$ is the residue of $S_1$ at $\lambda=i\sqrt{\nu_m + \rho^2}$ for $m=0,1, 2, \cdots$.  
\begin{theorem}\label{Ramanujan-per}
Let $\alpha, \beta>0$ and let $a\in \mathcal H(A, p, \delta)$. Then the following holds:
\begin{enumerate}
\item The Fourier series \begin{equation}\label{fourier-series}
 f(t)=2\pi i\sum_{m=0}^\infty \frac{d_m}{c_m} \, a(i\sqrt{\nu_m + \rho^2})\Psi_{m}(t),
\end{equation} converges uniformly   on compact subsets of $\Omega_p:=\{t\in\C\mid |\Re t|<\frac{\pi}{2}, |\Im t|<p\}$ and hence holomorphic there.

\item Let $0\leq \sigma<\delta$. Then the function $f$ can also be expressed in the integral form for $t\in\R$ with $|t|<p$ as
\begin{equation}\label{integral-f}
f(it)=\int_{-\infty-i\sigma}^{\infty -i\sigma} \left(a(\lambda) b(\lambda) + a(-\lambda) b(-\lambda)\right)\varphi_{\lambda}(t)\,\frac{d\lambda}{c(\lambda)c(-\lambda)},
\end{equation}
where the function $b$ is defined by
\begin{equation}
\label{defn-b}
\frac{b(\lambda)}{c(\lambda)c(-\lambda)}=S_1(\lambda).
\end{equation}
The integrals defined above are independent of $\sigma$ and extends as a holomorphic function to a neighboorhood  $\{z\in\C\mid |\Re z|<\frac{\pi}{2}-A\}$ of $i\R$.

\item The extension of $f$ to $i\R$ satisfies, for all $\lambda\in \mathbb R$ 
\begin{equation}\label{ab-lambda}
\int_{\R}f(it) \, \varphi_\lambda(t) A(t)\,dt= a(\lambda) b(\lambda) + a(-\lambda) b(-\lambda).
\end{equation}   
\end{enumerate}
\end{theorem}

\begin{proof} We shall first prove the theorem for the case when $\nu_0 +\rho^2>0$. Proof of the other case, i.e. the case when $\nu_0 + \rho^2=0$ will follow similarly. 

So let us assume that $\nu_0 +\rho^2>0$.
We first prove $(1)$. Using Corollary \ref{est-psi-cor} and equation (\ref{estimate-residue})  we get that 
$$\begin{array}{lll}
&&\sum_{m=0}^\infty \frac{|d_m|}{|c_m|} |a(i\sqrt{\nu_m + \rho^2})| \, |\Psi_{m}\left((t+ is) \right)| \\ 
&\leq&
C \sum_{m=0}^\infty \, (\nu_m + \rho^2) Q(|\sqrt{\nu_m +\rho^2}|) e^{-p \sqrt{\nu_m + \rho^2}} e^{|s| \sqrt{\nu_m + \rho^2}} |P(t + is)|\\
&<& \infty \text{ if } |s|<p.

\end{array}$$
Here we have used the fact that $c_m \rightarrow \infty$ as $m\rightarrow \infty$.

Now we shall prove $(2)$. Let $R,\sigma >0$. Let $\Gamma_1$ be the straight line joining $(-R, -\sigma)$ and $(R, -\sigma)$, $\Gamma_2$ be a straight line joining $(R, -\sigma)$ and $(R, R)$, $\Gamma_3$ be a strainght line joining $(R, R)$ and $(-R, R)$, $\Gamma_4$ be a strainght line joining $(-R, R)$ and $(-R, -\sigma)$. Let $\Gamma=\Gamma_1\cup \Gamma_2\cup \Gamma_3\cup \Gamma_4$ be the rectangle with anticlockwise direction.

Let $$\begin{array}{lll}
I &=& \int_\Gamma a(\lambda) b(\lambda) \varphi_\lambda(t) \frac{1}{c(\lambda) c(-\lambda)}\,d\lambda \\ \\
&=& \int_{\Gamma_1} a(\lambda) \varphi_\lambda(t)\frac{b(\lambda)}{c(\lambda) c(-\lambda)}\,d\lambda + \int_{\Gamma_2} a(\lambda) \varphi_\lambda(t)\frac{b(\lambda)}{c(\lambda) c(-\lambda)}\,d\lambda \\  \\
&+& \int_{\Gamma_3} a(\lambda) \varphi_\lambda(t)\frac{b(\lambda)}{c(\lambda) c(-\lambda)}\,d\lambda+ \int_{\Gamma_4} a(\lambda) \varphi_\lambda(t)\frac{b(\lambda)}{c(\lambda) c(-\lambda)}\,d\lambda.
\end{array}$$

We claim that 
\begin{equation}
\nonumber
\int_{\Gamma_i} a(\lambda) \varphi_\lambda(t)\frac{b(\lambda)}{c(\lambda) c(-\lambda)}\,d\lambda \rightarrow 0
\end{equation}
as $R\rightarrow \infty$, for $i=2, 3, 4$. Suppose that  the claim is true.  We observe that the function \begin{equation}\nonumber
\frac{b(\lambda)}{c(\lambda) c(-\lambda)}=S_1(\lambda),
\end{equation} has simple poles at $\lambda= i \sqrt{\nu_m +\rho^2}$ for $m=0,1, 2, \cdots$  in side the rectangle $\Gamma$. Therefore from Cauchy's theorem it follows that 
$$\begin{array}{lll}
\int_{-\infty-i\sigma}^{\infty-i\sigma} a(\lambda) b(\lambda) \varphi_\lambda(t) \frac{1}{c(\lambda) c(-\lambda)}\,d\lambda 
&=& 2\pi i \sum_{m=0}^\infty a\left(i\sqrt{\nu_m +\rho^2}\right) \varphi_{i\sqrt{\nu_m + \rho^2}} (t) \text{ Res}_{\lambda=i\sqrt{\nu_m +\rho^2}} S_1(\lambda) \\ \\
&=&2\pi i \sum_{m=0}^\infty \frac{d_m}{c_m} \,a\left(i\sqrt{\nu_m +\rho^2}\right) \Psi_{m}(it).
\end{array}$$
We shall prove the above claim.
Using (\ref{estimate-S}) we have for $|t|<p$, $$\begin{array}{lll}
\int_{\Gamma_2} \left| a(\lambda)\varphi_\lambda(t) \frac{b(\lambda)}{c(\lambda) c(-\lambda)} \,d\lambda\right| &=& \int_{s=-\sigma}^R |a(R +is)|\, |\varphi_{R + is}(t)| \, |S_1(R+is)|\,ds\\
&\leq& \int_{s=-\sigma}^R e^{-ps + AR} e^{|st|} e^{-\frac{\pi}{2} R} (R^2 + s^2)\, ds \\
&=& e^{(A-\pi/2) R} \int_{s=-\sigma}^R e^{-sp + |s|\,|t|}  (R^2 + s^2) \,ds\\
&\rightarrow& 0 \text{ as } R\rightarrow \infty.

\end{array}$$
The last line follows as $A<\pi/2$.
On $\Gamma_3$ for $|t|<p$,
$$\begin{array}{lll}
\int_{\Gamma_3} \left| a(\lambda)\varphi_\lambda(t) \frac{b(\lambda)}{c(\lambda) c(-\lambda)}\, d\lambda\right| &=& \int_{s=-R}^R |a(s+iR)|\, |\varphi_{s+ iR}(t)| \, |S_1(s+iR)|\,ds\\
&\leq& \int_{s=-R}^R e^{-pR + A|s|} e^{R|t|} e^{-\pi/2 |s|} (R^2 + s^2)\, ds \\
&=& e^{-R(p-|t|)}\int_{s=-R}^R e^{(A-\pi/2)|s|} (R^2 + s^2)\,ds\\
&\rightarrow& 0 \text{ as } R\rightarrow \infty.

\end{array}$$

For $\Gamma_4$, for $|t|<p$,
$$\begin{array}{lll}
\int_{\Gamma_4} \left| a(\lambda)\varphi_\lambda(t) \frac{b(\lambda)}{c(\lambda) c(-\lambda)}\,d\lambda\right| &=& \int_{s=-\sigma}^R |a(-R +is)|\, |\varphi_{-R + is}(t)| \, |S_1(-R+is)|\,ds\\
&\leq& \int_{s=-\sigma}^R e^{-ps + AR} e^{|st|} e^{-\pi/2 R} (R^2 + s^2)\, ds \\
&=& e^{(A-\pi/2) R} \int_{s=-\sigma}^R e^{-sp + |s|\,|t|} (R^2 + s^2)\,ds\\
&\rightarrow& 0 \text{ as } R\rightarrow \infty.

\end{array}$$
Using the same method we can show that the right hand side of (\ref{integral-f}) is independent of $0<\sigma<\delta$.

Now we consider 
$$\begin{array}{lll}
\int_{-\infty-i\sigma}^{\infty -i\sigma} \left| a(\lambda) b(\lambda) \varphi_\lambda(t+ is) \frac{d\lambda}{c(\lambda) c(-\lambda)}\right| &=&\int_{-\infty-i\sigma}^{\infty -i\sigma} \left|a(\lambda)\varphi_\lambda(t+ is) S_1(\lambda)\,d\lambda\right|\\
&\leq& \int_\R \left|a(y-i\sigma)\varphi_{y-i\sigma}(t+ is) S_1(y-i\sigma) (y^2 + \sigma^2)\,dy\right| \\
&\leq & \int_\R e^{p\sigma + A|y|} e^{|ys-t\sigma|}e^{-\pi/2 |y|} (y^2 + \sigma^2) \,dy.

\end{array}$$
This shows that the integral defined above is finite if $|s|<\pi/2-A$ and hence holomorphic when $|s|<\pi/2-A$.

We observe that the equation (\ref{integral-f}) is true for every $0<\sigma<\delta$ and the right hand side of (\ref{integral-f}) is independent of $\sigma$.  Hence using the fact that $c(-\lambda)=\overline{c(\lambda)}, \lambda\in\R$ we have 
\begin{equation}\label{integral-f-sigma0}
f(it)=\int_\R \left(a(\lambda) b(\lambda) + a(-\lambda) b(-\lambda)\right)\varphi_{\lambda}(t) |c(\lambda)|^{-2}\,d\lambda.
\end{equation}
To prove $(3)$ if we prove that the map \begin{equation}\nonumber
\lambda\mapsto a(\lambda) b(\lambda) + a(-\lambda) b(-\lambda),
\end{equation} is in $\mathcal S_2(\R)_e$, then using the inversion formula we will have (\ref{ab-lambda}). To show that the map \begin{equation}\nonumber
\lambda\mapsto a(\lambda) b(\lambda) + a(-\lambda) b(-\lambda),
\end{equation} is in $\mathcal S_2(\R)_e$, it is enough to show (using Cauchy's theorem) that the map is holomorphic on $\R+i[-\delta, \delta]$ and  for each $N\in\N$,
\begin{equation}
\label{ab-schwartz}
\sup_{\lambda\in \R+i[-\delta, \delta]} (1 + |\lambda|)^N |a(\lambda) b(\lambda)| <\infty,
\end{equation}
for some $\delta>0$.
We have \begin{equation}\nonumber
b(\lambda)=c(\lambda) c(-\lambda) S_1(\lambda).
\end{equation}
From the definition of $S_1$  and Corollary \ref{estimate-c} it follows that $b$ has simple pole at $\lambda=0$. Also we have $b(-\lambda)=-b(\lambda)$. Therefore the function  
\begin{equation}
\nonumber
a(\lambda) b(\lambda) + a(-\lambda) b(-\lambda)=b(\lambda) \left(a(\lambda)-a(-\lambda)\right)
\end{equation} is holomorphic around origin, as $b$ has simple pole at $\lambda=0$ and $a(\lambda)- a(-\lambda)$ has zero at $\lambda=0$. Hence it follows from Theorem \ref{c-holomorphic} that the map \begin{equation}\nonumber
\lambda\mapsto a(\lambda) b(\lambda) + a(-\lambda) b(-\lambda)\end{equation} is holomorphic on $\R+ i[-\delta, \delta]$ for some $\delta<\frac{1}{2}$. Using Corollary \ref{estimate-c} and (\ref{estimate-S}) we get that
\begin{equation}
\nonumber
\sup_{\lambda\in \R+i[-\delta, \delta]\setminus \{\text{ a nbd. of origin}\}} (1 + |\lambda|)^N |a(\lambda) b(\lambda)| <\infty,
\end{equation}
for some $\delta>0$. This completes the proof.
\end{proof}

\begin{remark}
We now consider the non-perturbed case, i.e. the case when $B(t)\equiv 1.$ In this case $\tilde{A}(t)=(\sin t)^{2\alpha+1} (\cos t)^{2\beta+1}$ and $A(t)=(\sinh t)^{2\alpha+1} (\cosh t)^{2\beta+1}$. The corresponding Sturm Liouville operators for $A$ and $\tilde{A}$ are well studied on $\R^+$ and $(0,\pi/2)$ respectively. Indeed the full spectral decomposition of $\mathcal{L}$ and $-L$ is known. 
Let $P_n^{\alpha,\beta}$ be a Jacobi polynomial of order $(\alpha,\beta)$ of degree $n.$

The operator $-l$ can be explicitly written as  $$-l=-\frac{d^2}{dt^2} + (\alpha^2-\frac{1}{4})\cot ^2 t + (\beta^2-\frac{1}{4})\tan ^2 t .$$ It is known (see \cite[p.67, sec4.24]{Szego}) that $u_n(t)=\sqrt{\tilde{A}(t)}P_n^{\alpha,\beta}(\cos 2t)$
are the eigenfunctions of $-l$ with eigenvalues $\nu_{n}=(2n+\rho)^2-\alpha^2-\beta^2 +\frac{1}{2},$ where $\rho= \alpha+\beta+1$ and $n\geq 0.$ Therefore in this case the eigenfunction  $\Psi_n$ of $-L$ reduces to a normalising factor times the Jacobi polynomial $P_n^{\alpha, \beta}(\cos 2t)$  with eigenvalue $\nu_{n}=(2n+\rho)^2-\alpha^2-\beta^2 +\frac{1}{2}$.
Hence the  relation (\ref{relation-phi-psi}) becomes
\begin{equation}
\label{relation-phi-psi-}
P_n^{\alpha, \beta}(\cos 2t)=c'_n\varphi_{i\sqrt{\nu_n+ \rho^2}}(it), \text{ on } (0, \frac{\pi}{2}),
\end{equation}
with $\rho=\alpha + \beta +1$. Then we define the function $b$ as in (\ref{defn-b}) and state the Ramanujan master's theorem as in Theorem \ref{Ramanujan-per}. We conclude here the function $b$ doesnot come out to be exactly $P(\lambda)\left(\sin{\frac{\pi}{2} (\lambda-\rho)}\right)^{-1}$ (where $P$ is a polynomial) as in \cite[Theorem 5.1]{Olafsson-2} (when restricted to rank one case) is because $\mathcal{L}$ differs from the Laplace Beltrami operator considered in \cite{Olafsson-2} by a constant dependent on $\alpha,\beta$ times $I$ (Identity operator), which makes $\nu_n +\rho^2$ a complete square for all $n$ in their case.

\end{remark}

\section{The case when $\alpha$ or $\beta\leq 0$}
 Let us consider the case when $ \alpha~~~ \text{or}~~~  \beta\in(-\frac{1}{2},0]$. We know from Theorem \ref{spectral} that the eigenvalues $\nu_n$ of $-L$ goes to $+\infty$ as $n\rightarrow \infty$. In the case of $\alpha, \beta>0$, all the eigenvalues $\nu_n$'s are non-negative. But in general (i.e. for $\alpha$ or $\beta\leq 0$) all of the eigenvalues may not be non-negative. Let $n_0$ be the smallest non negative integer such that $\nu_m $ is nonnegative for all $m\geq n_0$. Let $m_0$ be the largest non negative integer such that $\nu_{m_0} + \rho^2<0.$ It is easy to see that we can write $\sqrt{-\left(\nu_m + \rho^2\right)}=\beta_m$, $\beta_m>0$ for $m\leq m_0.$  
 
 We define the following functions:
\begin{equation}
S(z)= \left\{\begin{array}{lll} \pi z \prod_{n=n_0}^\infty \left(1 + \frac{z^2}{\nu_n+\rho^2}\right) &\text { if } & \nu_{n_0}+\rho^2>0,\\
\pi z^3 \prod_{n=n_0}^\infty \left(1 + \frac{z^2}{\nu_n+\rho^2}\right) &\text { if } & \nu_{n_0}+\rho^2=0. \end{array}\right.
\end{equation} 
 and
 \begin{equation}
 S_1(z)=\frac{z^2}{S(z)}.
 \end{equation}
 Then the conclusion of the Theorem \ref{defn-S} will remain same. Now we  state Ramanujan's master theorem in this case as follows and its proof is similar to the proof of Theorem \ref{Ramanujan-per}.

\begin{theorem}\label{Ramanujan-per-1}
Let $ \alpha~~~ \text{or}~~~  \beta\in(-\frac{1}{2},0]$ and let $a\in \mathcal H(A, p, \delta)$. Then the following holds:
\begin{enumerate}
\item The Fourier series \begin{equation}\label{fourier-series-1}
 f(t)=2\pi i\sum_{m=0}^\infty \frac{d_m}{c_m} \, a(i\sqrt{\nu_m + \rho^2})\Psi_{m}(t),
\end{equation} converges uniformly   on compact subsets of $\Omega_p:=\{t\in\C\mid |\Re t|<\frac{\pi}{2}, |\Im t|<p\}$ and hence holomorphic there.

\item Let $0\leq \sigma<\delta$. Then the function $f$ can also be expressed in the integral form for $t\in\R$ with $|t|<p$ as
\begin{equation}\nonumber
f(it)=\int_{-\infty-i\sigma}^{\infty -i\sigma} \left(a(\lambda) b(\lambda) + a(-\lambda) b(-\lambda)\right)\varphi_{\lambda}(t)\,\frac{d\lambda}{c(\lambda)c(-\lambda)} 
\end{equation}
\begin{equation}\nonumber  \hspace{1.2in} + 2\pi i \sum_{m= m_0+1}^{n_0-1} \frac{d_m}{c_m} a(i\sqrt{\nu_m + \rho^2})\Psi_{m}(it)\\
+2\pi i \sum_{m=0}^{m_0} \frac{d_m}{c_m} a(\beta_m) \Psi_{m}(it),
\end{equation}

where the function $b$ is defined by
\begin{equation}
\label{defn-b-1}
\frac{b(\lambda)}{c(\lambda)c(-\lambda)}=S_1(\lambda).
\end{equation}
The integrals defined above are independent of $\sigma$ and extends as a holomorphic function to a neighbourhood  $\{z\in\C\mid |\Re z|<\frac{\pi}{2}-A\}$ of $i\R$.

\item The extension of $f$ to $i\R$ satisfies, for all $\lambda\in \mathbb R$ 
\begin{equation}\nonumber
\int_{\R}\left( f(it)-2\pi i \sum_{m= m_0+1}^{n_0-1} \frac{d_m}{c_m} a(i\sqrt{\nu_m + \rho^2})\Psi_{m}(it)\\
-2\pi i \sum_{m=0}^{m_0} \frac{d_m}{c_m} a(\beta_m) \Psi_{m}(it) \right)\, \varphi_\lambda(t) A(t)\,dt
\end{equation}
\begin{equation}\nonumber
= a(\lambda) b(\lambda) + a(-\lambda) b(-\lambda).
\end{equation}   
\end{enumerate}
\end{theorem}

\section{Appendix}

\noindent{\bf Proof of Theorem \ref{est-phi}:}
In this subsection we prove Theorem \ref{est-phi}. The proof is similar to \cite[Theorem 1.2]{Brandolini-Gigante}. To prove the theorem we need the following preliminaries: 

Let $J_\alpha, Y_\alpha$ be the Bessel functions of first and second kind respectively. Also let $H_\alpha^{(1)}$ and $H_\alpha^{(2)}$ be the Hankel functions defined by
\begin{equation}
\nonumber
H_\alpha^{(1)}(x)= J_\alpha(x) + i Y_\alpha(x), H_\alpha^{(2)}(x)= J_\alpha(x) - i Y_\alpha(x).
\end{equation}
Also let \begin{equation}
\mathcal J_\alpha(x)=\sqrt{x} J_\alpha(x), \mathcal Y_\alpha(x)=\sqrt{x} Y_\alpha(x), \mathcal H_\alpha^{(1)}(x)=\sqrt{x} H_\alpha^{(1)}(x),  \mathcal H_\alpha^{(2)}(x)=\sqrt{x} H_\alpha^{(2)}(x).
\end{equation}
Let $\Omega=\{t+ is\mid |s|<\pi/2\}$. Also let \begin{equation}
\nonumber
\langle x \rangle=\frac{|x|}{1 + |x|}.
\end{equation}
To prove the Theorem \ref{est-phi}, the following two theorems are needed:
\begin{theorem} $($\cite[Theorem 10.1, p. 219]{Olver}$)$ \label{est-integraleqn}
Let $\alpha, \beta\in\C$ and let $\mathcal P$ be a finite chain of $R_2$ arcs in complex plane joining $\alpha$ and $\beta$.
Let \begin{equation}
\label{integral-equation}
h(\xi)=\int_\alpha^\xi K(\xi, v) \left(\psi_0(v) h(v) + \varphi(v) J(v)\right)\, dv,
\end{equation}
where
\begin{enumerate}
\item the path of integration lies on $\mathcal P$,
\item the real/complex valued functions $J(v), \varphi(v), \psi_0(v)$ are continuous except for a finite number of discontinuity,
\item the real/complex valued kernel $K(\xi, v)$ and its first two partial $\xi$ derivatives are continuous function on two variables $\xi, v \in \mathcal P$, (here all differentiation with respect to $\xi$ are perfomed along $\mathcal P$).
\item The kernel satisfies 
\begin{enumerate}
\item $K(\xi, \xi)=0$,
\item $|K(\xi, v)|\leq p_0(\xi) q(v)$,
\item $\left|\frac{\partial K}{\partial \xi}(\xi, v)\right| \leq p_1(\xi) q(v)$,
\item $\left|\frac{\partial^2 K}{\partial \xi^2}(\xi, v)\right| \leq p_2(\xi) q(v)$,
\end{enumerate}
for all $\xi\in \mathcal P$ and $v\in (\alpha, \xi)_\mathcal P$, for some continuous functions $p_0, p_1, p_2, q$. Here 
$ (\alpha, \xi)_\mathcal P$ denotes the part of $\mathcal P$ lying between $\alpha$ and $\xi$.

\item Also let the functions 
\begin{equation}\nonumber
\Phi(\xi)=\int_\alpha^\xi |\varphi(v)\,dv|, \Psi_0(\xi)=\int_\alpha^\xi |\psi_0(v)\,dv|
\end{equation}
converges.
\item Let \begin{equation}
\nonumber
\kappa=\sup_{\xi\in\mathcal P}\{q(\xi)|J(\xi)|\} \text{ and } \kappa_0=\sup_{\xi\in\mathcal P}\{p_0(\xi)q(\xi)\}
\end{equation}
are finite.
\end{enumerate}

Then the integral equation (\ref{integral-equation}) has a unique solution $h$ which is continuously differentiable in $\mathcal P$ and satisfies
\begin{equation}
\frac{|h(\xi)|}{p_0(\xi)}, \frac{|h'(\xi)|}{p_1(\xi)} \leq \kappa \Phi(\xi) \exp{\left(\kappa_0 \Psi_0(\xi)\right)}.
\end{equation}
\end{theorem}

\begin{theorem} $($\cite[Appendix, Lemma A.1]{Brandolini-Gigante}$)$ \label{integral-equation-1}
Let $u_1$ and $u_2$ be two linearly independent solutions of the equation \begin{equation}
\nonumber
u'' + p_1 u' + p_2u=0,
\end{equation}
on $\R$.
If $\varphi$ is a $C^2$ solution of the integral equation
\begin{equation}\nonumber
u(t)=-\int_{t_0}^t \frac{u_1(t)u_2(s)-u_2(t)u_1(s)}{u_1(s) u_2'(s)-u_2(s) u_1'(s)}\left(\psi_0(s) u(s) + J(s)\phi(s) \right)\,ds
\end{equation}
then $\varphi$ is a solution of 
\begin{equation}
\nonumber
u''+ p_1u' + p_2 u=\psi_0 u + J\phi.
\end{equation}
\end{theorem}

\begin{proof}[Proof of Theorem \ref{est-phi}] \label{proof-est-phi} We first assume that $\lambda\not=0$.
From the asymptotic expansion of $\varphi_\lambda$ (see \cite[p. 219]{Brandolini-Gigante}) we have
\begin{equation}
\label{assymp-sum-phi}
\varphi_\lambda(t)=\sum_{m=0}^M a_m(t) \frac{\mathcal J_{\alpha+ m}(\lambda t)}{\sqrt{A(t)} \lambda^{m+ \alpha+ \frac 12}} + \frac{R_M(\lambda, t)}{\sqrt{A(t)}},
\end{equation}
where $a_m$ are holomorphic and $M\geq 0$.

The function $R_M(\lambda, t)$ satisfies (see \cite[(1.3)]{Brandolini-Gigante})
\begin{equation}
\label{eqn-RM}
\frac{d^2R_M}{dt^2} + \left(\lambda^2-\frac{\alpha^2-\frac 14}{t^2}\right)R_M=G(t)R_M + 2 a_{M+1}'(t) \frac{\mathcal J_{\alpha+ M}(\lambda t)}{\lambda^{M + \alpha + \frac 12}}.
\end{equation}
Let $\Omega_1=\Omega\setminus (-\infty, 0]$.
From Theorem \ref{integral-equation-1}, it follows that a solution of the following integral equation 
\begin{equation}
\label{integral-RM}
R_M(\lambda, t)=-\pi\int_0^t\frac{\mathcal J_\alpha(\lambda t) \mathcal Y_\alpha(\lambda s)-\mathcal J_\alpha(\lambda s) \mathcal Y_\alpha(\lambda t)}{2\lambda} \left(G(s) R_M(\lambda, s) + 2 a_{M+1}'(s)\frac{\mathcal J_{\alpha+ M}(\lambda s)}{\lambda^{M+\alpha+ \frac 12}}\right)\,ds
\end{equation}
also satisfies (\ref{eqn-RM}). As shown in \cite[p. 222]{Brandolini-Gigante} $R_M(\lambda, t)$ is the solution of the above integral equation (\ref{integral-RM}), which satisfies the required Cauchy Condition.
Let $t_0\in (0, \infty)$. Then 
\begin{equation}
\nonumber
R_M(\lambda, t)=-\pi\int_{t_0}^t\frac{\mathcal J_\alpha(\lambda t) \mathcal Y_\alpha(\lambda s)-\mathcal J_\alpha(\lambda s) \mathcal Y_\alpha(\lambda t)}{2\lambda} \left(G(s) R_M(\lambda, s) + 2 a_{M+1}'(s)\frac{\mathcal J_{\alpha+ M}(\lambda s)}{\lambda^{M+\alpha+ \frac 12}}\right)\,ds
\end{equation}
\begin{equation}\nonumber
\hspace{2in}+ R_M(\lambda, t_0).
\end{equation}
Since both side of the equation above is holomorphic in $\Omega_1$, we have for all $\xi\in\Omega_1$,
\begin{equation}
\nonumber
R_M(\lambda, \xi)=-\pi\int_{t_0}^\xi\frac{\mathcal J_\alpha(\lambda \xi) \mathcal Y_\alpha(\lambda s)-\mathcal J_\alpha(\lambda s) \mathcal Y_\alpha(\lambda \xi)}{2\lambda}\left(G(s) R_M(\lambda, s) + 2 a_{M+1}'(s)\frac{\mathcal J_{\alpha+ M}(\lambda s)}{\lambda^{M+\alpha+ \frac 12}}\right)\,ds
\end{equation}
\begin{equation}\nonumber
\hspace{2in}+ R_M(\lambda, t_0).
\end{equation}
Let $$\begin{array}{lll}
    K(\xi, s)&=& -\pi \frac{\mathcal J_\alpha(\lambda \xi) \mathcal Y_\alpha(\lambda s)-\mathcal J_\alpha(\lambda s) \mathcal Y_\alpha(\lambda \xi)}{2\lambda}\\ \\
    &=& -\pi i \frac{\mathcal H_\alpha^{(1)}(\lambda \xi) \mathcal H_\alpha^{(2)}(\lambda s)-\mathcal H_\alpha^{(1)}(\lambda s) \mathcal H_\alpha^{(2)}(\lambda \xi)}{4\lambda}.
    \end{array}$$Using the estimates of Bessel and Hankel functions (as in \cite{Brandolini-Gigante}) we get,
    $$\left|K(\xi, s)\right|\leq \left\{ \begin{array}{lll}
    \frac{C}{|\lambda|} \langle \lambda\xi\rangle^{|\alpha| + \frac 12} \langle \lambda s\rangle^{-|\alpha| + \frac 12} e^{|\Im(\lambda \xi)-\Im(\lambda s)|} & \text{ for } & \alpha\not=0,\\ \\
    \frac{C}{|\lambda|} \langle \lambda\xi\rangle^{ \frac 12} \langle \lambda s\rangle^{\frac 12} \log(\frac{2}{\langle \lambda s\rangle}) e^{|\Im(\lambda \xi)-\Im(\lambda s)|} & \text{ for } & \alpha=0.
    \end{array}\right.$$
    Also we have $$\left|\frac{\partial}{\partial \xi}   K(\xi, s)\right|\leq \left\{ \begin{array}{lll}
    C \langle \lambda\xi\rangle^{|\alpha| -\frac 12} \langle \lambda s\rangle^{-|\alpha| + \frac 12} e^{|\Im(\lambda \xi)-\Im(\lambda s)|} & \text{ for } & \alpha\not=0,\\ \\
    C \langle \lambda\xi\rangle^{-\frac 12} \langle \lambda s\rangle^{\frac 12} \log(\frac{2}{\langle \lambda s\rangle}) e^{|\Im(\lambda \xi)-\Im(\lambda s)|} & \text{ for } & \alpha=0.
    \end{array}\right.$$
    For $\Im(\lambda \xi)>\Im(\lambda s)>0$, we let 
    \begin{enumerate}
    \item $p_0(\xi)=\frac{C}{|\lambda|} \langle \lambda\xi\rangle^{|\alpha| + \frac{1}{2}}e^{\Im(\lambda\xi)}$,
    \item $q(s)=\left\{\begin{array}{lll}
    \langle \lambda s\rangle^{-|\alpha| + \frac{1}{2}}e^{-\Im(\lambda s)} & \text{ for } \alpha\not=0\\ \\
 \langle \lambda s\rangle^{\frac{1}{2}}\log(\frac{2}{\langle \lambda s \rangle}) e^{-\Im(\lambda s)} & \text{ for } \alpha=0    
    \end{array}\right.$
    \item $p_1(\xi)=\langle \lambda\xi\rangle^{|\alpha| -\frac 12} e^{\Im(\lambda \xi)}$,
    \item $\psi_0(s)=G(s)$, $J(s)=\frac{1}{q(s)}$ and $\varphi(s)=2 a_{M+1}'(s) \frac{\mathcal J_{\alpha+ M}(\lambda s)}{\lambda^{M+\alpha + \frac{1}{2}}} q(s)$.
    \end{enumerate}
Therefore 
\begin{enumerate}
\item $\kappa_0:=\sup\{p_0(\xi) q(\xi)\}=\frac{C}{|\lambda|}$,
\item $\kappa:=\sup\{q(\xi) J(\xi)\}=1$.
\end{enumerate}
We have $$
|\varphi(s)| \leq \left\{ \begin{array}{lll} 
C |a_{M+1}'(s)| \,|\lambda|^{-M-\alpha-\frac 12} \langle \lambda s\rangle^{\alpha- |\alpha| +M+ 1} & \text{ for }& \alpha\not=0\\ \\
C |a_{M+1}'(s)| \, |\lambda|^{-M-\frac 12} \langle \lambda s\rangle^{M+ 1}\log(\frac{2}{\langle\lambda s\rangle}) & \text{ for }& \alpha=0
\end{array}
\right.$$
We know that $|a_{M+1}'(s)|\leq C s^M$ for $s<1$ and $a_{M+1}'\in L^1 ((1, \infty) )$.
Therefore, we have
$$
|\Phi(\xi)| \leq \left\{ \begin{array}{lll} 
C |\lambda|^{-M-\alpha-\frac 12} \langle \lambda \xi\rangle^{\alpha- |\alpha| +M+ 1} \langle \xi\rangle ^{M} & \text{ for }& \alpha\not=0\\ \\
C |\lambda|^{-M-\frac 12} \langle \lambda \xi\rangle^{M+ 1}\log(\frac{2}{\langle\lambda \xi\rangle})\langle \xi\rangle ^{M} & \text{ for }& \alpha=0
\end{array}
\right.$$

Hence by Theorem \ref{est-integraleqn}  we have  $$\begin{array}{lll}
\left| R_M(\lambda, \xi)\right| &\leq & C  |\lambda|^{-M-\alpha-\frac 32} \langle \lambda\xi \rangle^{M+ \alpha+ \frac32} e^{\Im(\lambda\xi)}\exp\{\left(\frac{C}{|\lambda|}\left| \int_0^\xi G(s)\,ds\right|\right)\} \\ \\
&\leq&  C |P(\xi)| \, e^{\Im(\lambda \xi)}  \exp\left(\frac{C}{|\lambda|}\left| \int_0^\xi G(s)\,ds\right|\right),

\end{array}$$
 for some polynomial $P$ for the case $\alpha\not= 0$. Then as in the argument \cite[Remark 1.3]{Brandolini-Gigante} we can improve the inequality above as
\begin{equation}
\label{est-R_M}
\left| R_M(\lambda, \xi)\right| \leq  C |P(\xi)| \, e^{\Im(\lambda \xi)}  \exp\left(\frac{C|\xi|}{1 + |\lambda \xi|}\left| \int_0^\xi G(s)\,ds\right|\right),
\end{equation} 
 Similar estimate also holds for $\alpha=0$. Therefore if $G$ is in $L^1$ in every direction, it follows that for all $\xi\in \Omega_1$
 \begin{equation}
 \nonumber
 |R_M(\lambda, \xi)|\leq C |P(\xi)|\, e^{\Im(\lambda\xi)}.
 \end{equation}
We can do the similar technique to the domain $\Omega_2=\Omega\setminus [0, \infty)$ and get the similar estimate for the domain $\Omega_2$.
 Hence we have from (\ref{assymp-sum-phi}) that
 \begin{equation}
 \nonumber
 |\varphi_\lambda(\xi)| \leq C\, |P(\xi)| \, e^{|\Im(\lambda\xi)|},
 \end{equation}
for all $\xi\in\Omega$.
\end{proof}

\vspace{.5in}

\noindent {\bf Holomorphic ODE:} We state the following theorem about the holomorphicity of solutions of a differential equation (\cite[Theorem 1.4]{Milicic}). This is used in the proof of the Lemma \ref{holo-extn-phi}.
\begin{theorem}\label{diffequation-holo}
Let $\Omega$ be a simply connected region in $\C$ and $z_0\in\C$.  Also let $a_1, a_2, \cdots, a_n$ be holomorphic functions on $\Omega$. For any complex numbers $y_0, y_1, \cdots, y_n$, there exists a unique holomorphic function $y$ on $\Omega$ such that
\begin{equation}
\frac{d^ny}{dx^n} + a_1\frac{d^{n-}y}{dx^{n-1}} + \cdots + a_{n-1}\frac{dy}{dx} + a_n y=0,
\end{equation}
and \begin{equation}
y(z_0)=y_0, y^{(1)}(z_0)=y_1, \cdots, y^{(n-1)}(z_0)=y_{n-1}.
\end{equation}
\end{theorem}

\vspace{.5in}

\noindent{\bf Singular Sturm Liouville operator:} In this subsection we give few well known preliminaries of Sturm-Liouville's operator. 

Let us define a singular Sturm Liouville operator on an interval $I=(a,b)$ by \begin{equation}\label{sturmliouville}My:=\frac{1}{w}\left(-(py')'+ qy\right),\end{equation} 
where $p,q,w:I\rightarrow\R,$ $p,w>0$ a.e. and $\frac{1}{p},q,w\in L^1_{\text{loc}}(I).$ The operator
$M$ is called {\em non-oscillatory} at $a$ if there exists a solution $My =\lambda y$ such that $y\neq 0$ in the $(a,a+\delta)$ for some $\delta>0$ and some $\lambda\in\R.$ Similar definition is for the other end point $b.$

For non oscillatory end points, Niessen and Zettl (in \cite{Niessen}) have completely characterised all the self adjoint extensions of the Sturm Liouville operator $M$ on $L^2((a,b),w(t)dt)$
 with explicit boundary conditions at $a$ and $b.$ 
 In  \cite[Theorem 4.2]{Niessen}, Niessen and Zettl obtain a Friedrich's extension of a Sturm Liouville operator $M$ as defined in \ref{sturmliouville}  on $(a,b)$ by defining the boundary conditions in terms of the principal solution at both end points. More precisely if $u_a$ and $u_b$ are any two principal solutions at $a$ and $b$ respectively, satisfying $Mu_j=\lambda_j u_j$ for $j=a,b$, $M$ is self adjoint extension on $L^2((a,b),w(t)dt)$ defined on the domain 
\[
\mathcal{D}=\left\{
y\in\mathcal{D}(M):\left[  y,u_a\right]  _{M}\left(
a\right)  =\left[  y,u_b\right]  _{M}\left(
b\right)  =0\right\}
\]
 The domain is independent of $u_j$ and $\lambda_j,j=a,b.$

  We say $u_a$ is a {\em principal solution} at $a$ if $u_a$ is non zero in a right neighbourhood of $a$ and for any other solution $y$ of $My=\lambda_a y$ on $(a,b),$ $u_a(t)=o(y(t))$ as $t\rightarrow a^+.$ It is known that a principle solution at $a$ of the equation $My=\lambda_a y$ is unique upto multiplicative constant. When $M$ is non-oscillatory at $a$ and $b,$ principal solutions do exist at $a$ and $b$ respectively.  
 
 If $M$ is a limit point case at $a$ i.e. only one solution of $Mu=\lambda u$ lies in $L^2
 (a,a+\epsilon)$ for some $\epsilon>0$  then we don't require any boundary condition at $a.$ This classification is independent of $\lambda.$ For further details see \cite{Niessen}.

\end{document}